\documentclass{amsart}
\numberwithin{equation}{section}

\newtheorem{thm}{Theorem}[section]
\newtheorem{lem}[thm]{Lemma}

\newtheorem{defin}[thm]{Definition}

\newtheorem{remark}[thm]{Remark}

\begin{document}

\begin{center}
\textbf{{\large {\ INVERSE PROBLEM FOR THE SUBDIFFUSION EQUATION WITH FRACTIONAL CAPUTO DERIVATIVE }}}\\[0pt]
\medskip \textbf{Ravshan Ashurov$^{1}$ and Marjona Shakarova$^{2}$}\\[0pt]
\textit{ashurovr@gmail.com\\[0pt]}
\medskip \textit{\ $^{1}$ Institute of Mathematics, Academy of Science of Uzbekistan}

\medskip \textit{$^{2}$ AU Engineering School, Akfa University, 264, Milliy Bog Str., Tashkent 111221, Uzbekistan}

\textit{shakarova2104@gmail.com\\[0pt]}

\medskip \textit{\ $^{1}$ Institute of Mathematics, Academy of Science of Uzbekistan}

\end{center}

\textbf{Abstract}: The inverse problem of determining the right-hand side of the subdiffusion equation with the fractional Caputo derivative is considered. The right-hand side of the equation has the form $f(x)g(t)$ and the unknown is function $f(x)$. The condition $ u (x,t_0)= \psi (x) $ is taken as the over-determination condition, where $t_0$ is some interior point of the considering domain and $\psi (x) $ is a given function. It is proved by the Fourier method that under certain conditions on the functions $g(t)$ and $\psi (x) $ the solution of the inverse problem exists and is unique. An example is given showing the violation of the uniqueness of the solution of the inverse problem for some sign-changing functions $g(t)$. It is shown that for the existence of a solution to the inverse problem for such functions $g(t)$, certain orthogonality conditions for the given functions and some eigenfunctions of the elliptic part of the equation must be satisfied.


\section{Introduction}

Let $\rho\in(0,1] $ be a fixed number.
Consider the following initial-boundary value problem
\begin{equation}\label{prob1}
\left\{
\begin{aligned}
& D_t^\rho u(x,t)-\Delta u(x,t)=F(x,t)\equiv f(x)g(t),\quad x\in \Omega, \quad t\in (0,T],\\
&u{(x,t)}|_{\partial\Omega}=0, \\
&u(x,0)=\varphi(x), \quad x\in \Omega.
\end{aligned}
\right.
\end{equation}
Here $f(x)$, $g(t)$ and $\varphi(x)$ are continuous functions in the domain $\Omega\subset \mathbb{R}^N$ and $
D_t^\rho h(t)$ stands for the Caputo fractional derivative (see e.g. \cite{Pskhu}, p.14)
\[
D_t^\rho h(t)=  \int\limits_0^t \omega_{1-\rho}(t-s) \frac{d}{ds}h(s)ds, \quad \omega_{\rho}(t)=\frac{t^{\rho-1}}{\Gamma(\rho)},
\]
where $\Gamma(\rho)$ is the gamma function. If we first integrate and then differentiate, then we get the Riemann-Liouville derivative.

It should be noted that if $\rho=1$, then both the Caputo derivative and the Riemann-Liouville derivative coincide with the classical first-order derivative. Therefore, if $\rho=1$, then problem (\ref{prob1}) coincides with the usual initial-boundary value problem for the diffusion equation.

Problem (\ref{prob1}) is also called the \textit{forward problem}.
The main purpose of this study is the inverse problem of determining the right-hand side of the equation, namely function $f(x)$. To solve the inverse problem one needs an extra condition. Following A.I. Prilepko and A.B. Kostin \cite{Pr} and K.B. Sabitov \cite{Sab} (see also \cite{Sab2}) we consider the additional condition in the form:
\begin{equation}\label{ad}
u (x,t_0) = \psi (x) , \quad x\in \Omega,\\
\end{equation}
where $t_0$ is a given fixed point of the segment $(0, T]$.

Let us call the initial-boundary value problem (\ref{prob1}) together with the additional condition (\ref{ad}) \emph{ the inverse problem} of finding the part $f(x)$ of the right-hand side of the equation.

The authors usually set an additional condition (\ref{ad}) at the final time $t_0=T$ (see, e.g. \cite{Orl}, \cite{Tix} for classical diffusion equations and for subdiffusion equations see \cite{MS}, \cite{MS1}). The meaning of taking condition (\ref{ad}) at $t_0$ is that in some cases the uniqueness of the solution of the inverse problem is violated if $t_0=T$ and by choosing $t_0$ it is possible to achieve uniqueness in these cases as well.

We will be interested in{\it the classical solution} (we will simply call it a solution) of the problems under consideration, i.e. such solutions that themselves and all the derivatives involved in the equation are continuous, moreover, all the given functions are continuous and the equation is performed at each point. As an example, let us give the definition of the solution to the inverse problem.
 \begin{defin}\label{def} A pair of functions $\{u(x,t), f(x)\}$ with the properties
\begin{enumerate}
	\item
	$D_t^\rho u(x,t), \Delta u(x,t)\in C(\overline{\Omega} \times (0.T])$,
	\item$u(x,t)\in C(\overline{\Omega}\times [0.T])$,
	\item
$f(x)\in C(\overline{\Omega})$,
\end{enumerate}
and satisfying conditions
(\ref{prob1}), (\ref{ad})  is called \textbf{the
 solution} of the inverse problem.
\end{defin}

We note that in this definition the requirement of continuity in a closed domain of all derivatives of the solution appearing in (\ref{prob1}) was proposed by O.A. Ladyzhenskaya \cite{{Lad}}. The advantage of this choice is that the uniqueness of just such a solution is proved quite simply, moreover, the solution found by the Fourier method satisfies the above conditions.

Inverse problems of determining the right-hand side of various subdiffusion equations have been studied by a number of authors due to the importance of such problems for applications. However, it should be immediately noted that for the abstract case of the source function $F(x,t)$ there is no general theory yet (see the survey paper \cite{{Hand1}} and the literature therein). In all known works, the split source function $F(x,t)\equiv f(x)g(t)$ is considered and the methods of investigation depend on whether $f(x)$ or $g(t)$ is unknown.
It is somewhat more difficult to study the case when function $g(t)$ is unknown. For example, in the papers \cite{Ash1} and \cite{Ash2} the questions of finding the non-stationary source function $g(t)$ are studied. It should be noted that in these papers the over-determination condition is taken in a fairly general form: $B[u(\cdot, t)] = \psi(t)$, where $B$ is a linear continuous functional. In particular, one can take $u(x_0,t)$ or $\int_{\Omega}u(x,t) dx$ as $B[u(\cdot, t)]$. Finding the unknown function $g(t)$ for subdiffusion equations is studied in the articles \cite{Hand1} and \cite{Yama11}.

For subdiffusion and diffusion equations, the case $g(t)\equiv 1$ and the unknown is $f(x)$ has been studied by many authors (see, for example, \cite{Fur}-\cite{4}). We will mention only some of these articles.

Subdiffusion equations with an elliptic part as an ordinary differential expression are considered in the articles \cite{Fur},\cite{15}, \cite{16}. Authors of articles  \cite{20}, \cite{24} studied subdiffusion equations whose elliptic part is the Laplace operator or a second-order differential operator. The paper \cite{25} is devoted to study the inverse problem for a subdiffusion equation with the Caputo fractional derivative and an arbitrary elliptic selfadjoint differential operator. The authors of this paper proved the uniqueness and existance of a generalized solution.  The case of the Riemann-Liouville derivative considered in \cite{4}. Here the uniqueness and existance of a classical solution is proved. In the papers \cite{Fur} and \cite{24}, the fractional derivative in the subdiffusion equation is a two-parameter generalized Hilfer fractional derivative.

In \cite{AshF}, the authors considered the inverse problem of simultaneous determination of the order of the Riemann-Liouville fractional derivative and the source function in subdiffusion equations. Using the classical Fourier method, the authors proved that the solution to this inverse problem exists and is unique.

In the monograph by K. B. Sabitov \cite{Sab3} the solvability of forward and inverse problems for equations of mixed parabolic-hyperbolic type was studied.

We note some results obtained for the case $g(t)\not\equiv 1$.
For classical diffusion equations, such an inverse problem has been studied in detail (see the well-known monograph by S. Kabanikhin (\cite{Kab1}, Chapter 8,  see also \cite{Pr}, \cite{Sab}, \cite{Sab2}, \cite{Orl}, \cite{Tix}). Since the equation considered by us also covers the diffusion equation, we will dwell on these works in more detail at the end of  Section 4.

In the paper \cite{FN} the problem of finding the function $f(x)$ for an abstract subdiffusion equation with the Caputo derivative is studied.
To find function $f(x)$, the authors used the following additional condition $\int_0^T u(t)d\mu(t)=u_T$.

M. Slodichka et al. \cite{MS} and \cite{MS1} studied the uniqueness of a solution of the inverse problem for a subdiffusion equation, the elliptic part of which depends on time. It is proved that if function g(t) does not change sign, then the solution of the inverse problem is unique. It should be especially noted that in \cite{MS1} the authors constructed an example of a function $g(t)$ that changes sign in the domain under consideration, as a result of which the uniqueness of the solution of the inverse problem is lost.

It is well known that the considering inverse problem is ill-posed, i.e., the solution does
not depend continuously on the given data. Therefore, in the works of some authors, various regularization methods are proposed for constructing an approximate solution of the inverse problem (see, e.g., \cite{S}, \cite{Niu}). Thus, in paper \cite{S} the inverse problem for the fractional diffusion equation with the Riemann-Liouville derivative is considered. Assuming that solutions to the equation can be represented by a Fourier series, the authors applied the Tikhonov regularization method to find an approximate solution. Convergence estimates for exact and regularized solutions are presented for a priori and a posteriori rules for choosing parameters. In \cite{Niu}, similar questions are investigated for the stochastic fractional diffusion equation.

This work is devoted to the study of the forward problem (\ref{prob1}) and the inverse problem (\ref{prob1}), (\ref{ad}) of determining the right-hand side of the equation.
Let us list the main results of this paper.

1) First (in Section 3), we prove the existence and uniqueness theorem for the forward problem (\ref{prob1}) using the Fourier method. We present conditions on the initial function $\varphi(x)$ and on the right-hand side of the equation that ensure the validity of the application of the Fourier method. Due to the fact that the elliptic part of the equation is the Laplace operator, the conditions on the functions $f(x)$ and $g(t)$ turned out to be easier to check than in the case of a general elliptic operator (see, \cite{AshM});

2) Then (in Section 4), under a certain condition on function $g(t)$ (for example, the constant sign is sufficient), we prove the existence and uniqueness of a solution to the inverse problem. Further, we will show that if this condition is violated, then for the existence of a solution to the inverse problem, it is sufficient that the functions from the initial condition and the over-determination condition be orthogonal to some eigenfunctions of the Laplace operator with the Dirichlet condition;

3) An example of function $g(t)$ is constructed (in Section 4), for which the condition noted above is not satisfied and, as a result, the inverse problem has more than one solution.

The following Section 2 is auxiliary and contains definitions and well-known assertions necessary for further presentation. The section Conclusions completes this work.

\section{Preliminaries}

In this section, which has an auxiliary character, we define fractional powers of a self-adjoint extension of the Laplace operator, formulate a lemma, from the book of Krasnoselskii et al. \cite{K}, the fundamental result of V.A. Il'in \cite{Il} about the convergence of the Fourier coefficients and indicate some properties of the Mittag-Leffler function that we need.

Denote by $\{v_k(x)\}$ the complete system of orthonormal eigenfunctions in $L_2(\Omega)$ and by $\{\lambda_k\}$ the set of positive eigenvalues of the following spectral problem
\[
\left\{
\begin{aligned}
& -\Delta v(x)=\lambda v(x),\quad x\in \Omega,\\
& v(x)|_{\partial\Omega}=0. \\
\end{aligned}
\right.
\]

Let $\sigma$ be an arbitrary real number. Consider an operator $\hat{A}^\sigma$ acting in $L_2(\Omega)$ as:
\[
\hat{A}^{\sigma}g(x)=\sum\limits_{k=1}^\infty \lambda_k^{\sigma} g_k v_k(x), \quad g_k=(g,v_k),
\]
with the domain of definition
\[
D(\hat{A}^{\sigma}) = \{g \in L_2(\Omega) :\sum\limits_{k=1}^\infty \lambda_k^{2\sigma} |g_k|^2<\infty  \}.
\]
For elements of $D(\hat{A}^{\sigma})$ we introduce the norm
\[
||g||^2_\sigma =\sum\limits_{k=1}^\infty \lambda_k^{2\sigma} |g_k|^2=||\hat{A}^{\sigma}g||^2.
\]
If $A$ stands for the operator acting in $L_2(\Omega)$ as $Ag(x) = -\Delta g(x)$ with the domain of definition $D(A) = \{g \in C^2(\Omega) : g(x) = 0, x \in \partial\Omega \}$, then $\hat{A}\equiv\hat{A}^1$  is a self-adjoint extension of $A$ in $L_2(\Omega)$.

In our reasoning the following lemma from the book Krasnoselskii et al. \cite{K} plays an important role.
\begin{lem}\label{ml1} Let $\sigma>\frac{N}{4}$. Then operator $\hat{A}^{-\sigma}$ continuously maps the space $L_2(\Omega)$ into $C(\Omega)$, and moreover, the following estimate holds
\[
||\hat{A}^{-\sigma}g||_{C(\Omega)}\leq C||g||_{L_2(\Omega)}.
\]
\end{lem}
In order to prove the existence of solutions of forward and inverse problems, it is necessary to study the convergence of the following series:
\begin{equation}\label{ml2}
\sum\limits_{k=1}^\infty \lambda_k^\tau|h_k|^2, \quad \tau > \frac{N}{2},
\end{equation}
where $h_k$ are the Fourier coefficients of function $h(x)$. In the case of integers $\tau$, in the fundamental paper \cite{Il} by V.A. Il'in, conditions are obtained for the convergence of such series in terms of the membership of function $h(x)$ in the classical Sobolev spaces $W_2^k(\Omega)$. To formulate these conditions, we introduce the class $\hat{W}_2^1(\Omega)$ as the closure in the $W_2^1(\Omega)$ norm of the set of all functions that are continuously differentiable in $\Omega$ and vanish near the boundary of $\Omega$.

The theorem of V. A. Il'in states that, if function $h(x)$ satisfies the conditions
 \begin{equation}\label{ml3}
h(x)\in W_2^{\big[\frac{N}{2}\big]+1}(\Omega) \quad and \quad h(x),\Delta h(x),....,\Delta^{\big[\frac{N}{4}\big]}h(x) \in \hat{W}_2^1(\Omega),
\end{equation}
then the number series (\ref{ml8}) converges. Here $[a]$ denotes the integer part of the number $a$.
 Similarly, if in (\ref{ml8}) we replace $\tau$ by $\tau+2$, then the convergence conditions will have the form:

\begin{equation}\label{ml4}
h(x)\in W_2^{\big[\frac{N}{2}\big]+3}(\Omega)\quad and \quad h(x),\Delta h(x),....,\Delta^{\big[\frac{N}{4}\big]+1}h(x) \in \hat{W}_2^1(\Omega).
\end{equation}

For $0 < \rho < 1$ and an arbitrary complex number $\mu$, let $E_{\rho, \mu}(z)$ denote the Mittag-Leffler function with two parameters of the complex argument $z$:
\begin{equation}\label{ml}
E_{\rho, \mu}(z)= \sum\limits_{k=0}^\infty \frac{z^k}{\Gamma(\rho
k+\mu)}.
\end{equation}
If the parameter $\mu =1$, then we have the classical
Mittag-Leffler function: $ E_{\rho}(z)= E_{\rho, 1}(z)$.

Recall some properties of the Mittag-Leffler functions (see, e.g. \cite{Dzh66}, p. 134 and p. 136).

\begin{lem}\label{mll4} For any $t\geq 0$ one has
\begin{equation}\label{mL_1}
|E_{\rho, \mu}(-t)|\leq \frac{C}{1+t},
\end{equation}
where constant $C$ does not depend on $t$ and $\mu$.
\end{lem}
\begin{lem}\label{MLmonoton} (see \cite{Gor}, p. 47).
The classical Mittag-Leffler function of the negative argument $E_\rho(-t)$ is monotonically
decreasing function for all $0 <\rho < 1$ and
\[
0<E_{\rho} (-t)<1,\quad E_{\rho} (0)=1.
\]
\end{lem}

\begin{lem}\label{ml8} (see \cite{Dzh66}, formula (2.30), p.135 and \cite{AShZun}, Lemma 4). Let $\mu$ be an arbitrary complex number. Then the following asymptotic estimate holds
\[
E_{\rho, \mu}(-t)= \frac{t^{-1}}{\Gamma(\mu-\rho)} + O(t^{-2}), \quad t>1.
\]
\end{lem}

\begin{lem}\label{MLint} (see \cite{Gor}, formula (4.4.5), p. 61).
		Let $\rho > 0 $, $\mu>0$ and $\lambda \in C$. Then for all positive $t$ one has
		\begin{equation}\label{MLintFormula}
			\int\limits_0^t (t-\eta)^{\mu-1}\eta^{\rho-1}E_{\rho,\rho}(\lambda\eta^\rho)d\eta=t^{\mu+\rho-1} E_{\rho,\rho+\mu}(\lambda t^\rho).
		\end{equation}
	\end{lem}
\

\section{Well-posedness of forward problem (\ref{prob1})}

First we consider the following  problem for the homogeneous equation:

\begin{equation}\label{prob.y}
\left\{
\begin{aligned}
& D_t^\rho y(x,t)-\Delta y(x,t) =0, \quad (x,t)\in \Omega\times (0,T],\\
&y{(x,t)}|_{\partial{\Omega}}=0, \\
&y(x,0)=\varphi(x), \quad x \in \Omega,
\end{aligned}
\right.
\end{equation}
where $\varphi (x) $ is a given function.

\begin{thm}\label{main}
 Let function $\varphi (x) $ satisfy conditions (\ref{ml3}). Then problem (\ref{prob.y}) has a unique solution:
\begin{equation}\label{prob.s}
       y(x,t)=\sum\limits_{k=1}^{\infty} \varphi_k E_\rho(-\lambda_k t^\rho )v_k(x),
    \end{equation}
    where $\varphi_k$ are the Fourier coefficients of function $\varphi(x)$.
\end{thm}

\begin{proof} This theorem for a more general subdiffusion equation was proved in \cite{AshM}. We only mention the main points of the proof.

Obviously, (\ref{prob.s}) is a formal solution to problem (\ref{prob.y}) (see \cite{Pskhu}, p. 17, \cite{AshCab}).

Let us show that the operators $A = -\Delta$ and $D_t^\rho $ can be applied term-by-term to series (\ref{prob.s}) and the resulting series converges uniformly in $(x, t) \in  (\overline{\Omega}\times (0, T])$.

   If $S_j(x,t)$ is the partial sum of series (\ref{prob.s}), then
   \[
-\Delta S_j(x,t)=\sum\limits_{k=1}^{j} \lambda_k \varphi_k E_\rho(-\lambda_k t^\rho )v_k(x).
\]
   Using the equality
    \[
   \hat{A}^{-\sigma} v_k(x) = \lambda_k^{-\sigma}v_k (x),
    \]
with $\sigma > \frac{N}{4}$ and applying Lemma \ref{ml1} for $g(x) =-\Delta S_j(x,t)$, we have
    \[
 ||-\Delta S_j(x,t)||^2_{C(\Omega)}= ||\sum\limits_{k=1}^{j} \lambda_k \varphi_k E_\rho(-\lambda_k t^\rho )v_k(x)||^2_{C(\Omega)}\leq C \sum\limits_{k=1}^{j} \lambda^{2(\sigma+1)}_k  |\varphi_k E_\rho(-\lambda_k t^\rho )|^2. \]
Apply estimates (\ref{mL_1}), to obtain
 \[ ||-\Delta S_j(x,t)||^2_{C(\Omega)}\leq C \sum\limits_{k=1}^{j} \frac{\lambda^{2(\sigma+1)}_k  |\varphi_k|^2}{|1+\lambda_k  t^\rho|^2} \leq Ct^{-2\rho} \sum\limits_{k=1}^{j} \lambda^{2\sigma}_k  |\varphi_k|^2,\, \quad t > 0.     \]
Therefore if $\varphi (x)$ satisfies conditions (\ref{ml3}), then $ -\Delta y(x, t) \in C(\overline{\Omega}\times (0, T])$.

   From equation (\ref{prob.y}) one has
$D_t^\rho y(x, t) = \Delta y(x, t)$, $t > 0$, and hence we get $D_t^\rho y(x, t) \in C(\overline{\Omega}\times (0, T])$.

The uniqueness of the solution follows from the completeness of the system $\{v_k(x)\}$ in $L_2(\Omega)$ (see \cite{4}). We only note that it is important here that the derivatives of the solution involved in the equation are continuous up to the boundary of domain $\Omega$ (see Definition \ref{def}).  Nevertheless, below we give a proof of the uniqueness of a solution of the inverse problem in detail (see the proof of Theorem \ref{thmNotChange}).
\end{proof}
Now consider the following auxiliary initial-boundary value problem:
\begin{equation}\label{prob6}
    \left\{
    \begin{aligned}
        & D_t^\rho \omega(x,t)
    -\Delta\omega(x,t) =f(x)g(t),\,\, (x,t)\in \Omega \times(0,T],\\
        &\omega(x,t)|_{\partial{\Omega}}=0, \\
        & \omega(x,0) =0,\,\,  x\in {\Omega}.
    \end{aligned}
    \right.
\end{equation}

\begin{thm}\label{main1} Let $f(x)$ satisfy conditions (\ref{ml3}) and
$ g(t)\in C[0,T]$. Then problem (\ref{prob6}) has a unique solution
\begin{equation}\label{prob4}
\omega(x,t)=\sum\limits_{k=1}^\infty f_k \left [\int\limits_0^t \eta^{\rho-1} E_{\rho, \rho} (-\lambda_k  \eta^\rho ) g(t-\eta) d\eta \right]  v_k(x).
\end{equation}
where $f_k=(f,v_k)$.
\end{thm}

\begin{proof}
Again, as in the previous theorem, (\ref{prob4}) is a formal solution to problem (\ref{prob6}) (see \cite{Pskhu}, p. 17, \cite{AshCab}).

Let $S_j (x, t)$ be the partial sum of series (\ref{prob4}) and $ \sigma > \frac{N}{4}$. Repeating the above reasoning based on Lemma \ref{ml1}, we arrive at
    \[
    ||-\Delta S_j(x,t)||^2_{C(\Omega)}=	\bigg|\bigg| \sum\limits_{k=1}^\infty \lambda_k f_k \int\limits_0^t \eta^{\rho-1} E_{\rho, \rho} (-\lambda_k  \eta^\rho ) g(t-\eta) d\eta  v_k(x) \bigg|\bigg |^2_{C(\Omega)}\]
    \[
    \leq \bigg|\bigg|\hat{A}^{-\sigma} \sum\limits_{k=1}^\infty \lambda_k^{\sigma+1}f_k \int\limits_0^t \eta^{\rho-1} E_{\rho, \rho} (-\lambda_k  \eta^\rho ) g(t-\eta) d\eta   v_k(x) \bigg|\bigg |^2_{C(\Omega)}\]
    \[
    \leq \bigg|\bigg| \sum\limits_{k=1}^\infty \lambda_k^{\sigma+1} f_k\int\limits_0^t \eta^{\rho-1} E_{\rho, \rho} (-\lambda_k  \eta^\rho ) g(t-\eta) d\eta   v_k(x) \bigg|\bigg |^2_{L_2(\Omega)}\]
    (apply Parseval's equality to obtain)

    \[\leq C\sum\limits_{k=1}^\infty\bigg| \lambda_k^{\sigma+1}f_k \int\limits_0^t \eta^{\rho-1} E_{\rho, \rho} (-\lambda_k  \eta^\rho ) g(t-\eta) d\eta  \bigg |^2
    \]

  \[ \leq	C \sum\limits_{k=1}^\infty \bigg[  \lambda_k^{\sigma+1}|f_k| \max\limits_{0\leq t\leq T}|g(t)|
    \int\limits_0^t \eta^{\rho-1} E_{\rho, \rho} (-\lambda_k  \eta^\rho )d\eta  \bigg ]^2 \]
   (apply Lemma \ref{MLint} to get )
    \[ \leq	C \sum\limits_{k=1}^\infty \bigg[  \lambda_k^{\sigma+1}|f_k| \max\limits_{0\leq t\leq T}|g(t)|
    t^{\rho} E_{\rho, \rho+1} (-\lambda_k  t^\rho ) \bigg ]^2, \quad t>0. \]

    Lemma  \ref{mll4} implies
    \[||-\Delta S_j(x,t)||_{C(\Omega)}\leq	C\max\limits_{0\leq t\leq T}|g(t)|||f||_{\sigma}, \quad t > 0.\]

Hence $ -\Delta \omega(x,t) \in  C(\overline{\Omega}\times(0,T])$ and in particular
$ \omega(x,t) \in  C(\overline{\Omega}\times[0,T])$.
Further, from equation (\ref{prob6}) one has $ D_t^\rho S_j(x,t)=-\Delta S_j(x,t)+\sum\limits_{k=1}^j f_k g(t)v_k(x)$, $ t> 0$. Therefore, from
the above reasoning, we have $ D_t^\rho \omega(x,t) \in  C(\overline{\Omega}\times(0,T])$.

Again the uniqueness of the solution follows from the completeness of the system $\{v_k(x)\}$ in $L_2(\Omega)$.

\end{proof}

Now let us move on to solving the main problem (\ref{prob1}).  Note, if $y(x, t)$ and $\omega(x,t)$ are the solutions of problems (\ref{prob.y}) and (\ref{prob6}) correspondingly, then function $ u(x, t) = y(x,t) + \omega(x,t)$ is a solution to problem (\ref{prob1}). Therefore, we can use the already proven assertions and obtain the following result.

\begin{thm}\label{mainForward}
Let $\varphi (x)$, $f(x)$ satisfy conditions (\ref{ml3}) and $g(t)\in C[0, T]$. Then problem (\ref{prob6}) has a unique solution
\begin{equation}\label{probMain}
u(x,t)=\sum\limits_{k=1}^\infty \left [\varphi_k E_\rho(-\lambda_k t^\rho )+f_k \int\limits_0^t \eta^{\rho-1} E_{\rho, \rho} (-\lambda_k  \eta^\rho ) g(t-\eta) d\eta \right]  v_k(x).
\end{equation}
\end{thm}

\section{Well-posedness of the inverse problem (\ref{prob1}), (\ref{ad})}
We apply the additional condition (\ref{ad})  to equation (\ref{probMain}) and denote by $\psi_k$ the Fourier coefficients of  function $\psi(x): \psi_k  = (\psi, v_k )$. Then
\begin{equation}\label{EqFor_fk1}
\sum\limits_{k=1}^\infty f_k b_{k,\rho}(t_0)v_k(x)=
\sum\limits_{k=1}^\infty  \psi_k  v_k(x)-\sum\limits_{k=1}^\infty  \varphi_kE_\rho(-\lambda_k  t_0)  v_k(x),
\end{equation}   
where
\[
b_{k,\rho}(t)=\int\limits_0^{t} (t-s)^{\rho-1} E_{\rho, \rho} (-\lambda_k  (t-s)^\rho ) g(s) ds.
\]

From here, to find $f_k$, we obtain the following equation
\begin{equation}\label{EqFor_fk2}
f_k b_{k,\rho}(t_0)=
  \psi_k  -  \varphi_kE_\rho(-\lambda_k  t_0).
\end{equation}  
Of course the case $b_{k,\rho}(t_0)=0$ is critical. This can happen when $g(t)$ changes sign. The following example shows that for such $g(t)$  the uniqueness of the unknowns $f_k$ can be violated (see also \cite{MS1}).

 Example 1.  Consider the following  homogeneous inverse problem
 \begin{equation}\label{prob13}
\left\{
\begin{aligned}
& D_t^\rho u(x,t)-\Delta u(x,t) =f(x)g(t),\quad (x,t)\in \Omega \times(0,T],\\
&u(x,t)|_{\partial\Omega}=0,  \\
&u(x,0)=0, \quad x \in \Omega, \\
&u (x,t_0) = 0, \quad x\in \Omega.
\end{aligned}
\right.
\end{equation}  
 Take any eigenfunction $v$ of the  Laplace operator subject to homogeneous Dirichlet boundary
conditions, i.e.
$- \Delta v = \lambda v$ with $v(x)|_{\partial{\Omega}}=0$ and set $t_0 = 1$, $T(t) =t^\rho (1-t^b)$, $b>0$. Then, $u(x, t) = T(t)v(x)$ satisfies problem (\ref{prob13}) with
\[
f(x)= v(x) \,\,\,\text{and} \,\,\, g(t)=D_t^\rho T(t)+\lambda T(t).
\]
Then, besides the trivial solution $(u, f ) = (0, 0)$ to problem  (\ref{prob13}), we also have the following non-trivial solution
\[
u(x, t) = T(t)v(x),\,\, f(x)=v(x).
\]
It can be easily shown that, for example, for the parameters $b=0.1$ and $\rho=0.5$, the function $g(t)$ changes its sign.
Indeed, one has
\[
g(t) =\frac{\rho B(\rho,1-\rho)}{\Gamma(1-\rho)}-\frac{(b+\rho)t^\rho B(b+\rho,1-\rho)}{\Gamma(1-\rho)}+\lambda t^\rho (1-t^b),
\] 
and
\[
g(0) =0.5\Gamma(0.5)=\frac{\sqrt{\pi}}{2}> 0,\]

\[
g(1)=0.5\Gamma(0.5)-\frac{0.6 B(0.6,0.5)}{\Gamma(0.5)}=\frac{\sqrt{\pi}}{2}-\frac{ 6\Gamma(0.6)}{\Gamma(1.1)} < 0.
\] 

Let us divide the set of natural numbers $\mathbb{N}$ into two groups $K_{0,\rho}$ and $K_\rho$: $\mathbb{N}=K_\rho \cup K_{0,\rho}$, while the number $k$ is assigned to $K_{0,\rho}$, if $b_{k,\rho}(t_0)=0$, and if $b_{k,\rho}(t_0)\neq 0$, then this number is assigned to $K_\rho$. Note that for some $t_0$ the set $K_{0,\rho}$ may be empty, then $K_\rho=\mathbb{N}$. For example, if $g(t)$ is sign-preserving, then $K_\rho=\mathbb{N}$, for all $t_0$. 

The question naturally arises about the size of set $K_{0,\rho}$, i.e., how many elements does $K_{0,\rho}$  contain? As the authors of the paper \cite{Tix} noted, at least for $\rho=1$, the set $K_{0,1}$ can contain infinitely many elements. Indeed, in this case 
\[
b_{k,1}(t_0)=\int\limits_0^{t_0} e^{-\lambda_k  (t_0-s)} g(s) ds,
\]
and according to Müntz's theorem (see the monograph by S. Kaczmarz and H. Steinhouse \cite{Kac}, p. 103),  the set $K_{0,1}$  for some continuous functions $g(t)$ contains infinitely many elements (see also \cite{Le}, p. 107).

In the case of the diffusion equation, the criterion for the uniqueness of a solution of the inverse problem was studied in the papers cited above \cite{Pr}, \cite{Sab}, \cite{Sab2}, \cite{Orl}, \cite{Tix}. This criterion can be formulated as follows: The inverse problem has a unique solution if and only if: 
\begin{equation}\label{criterion1}
b_{k,1}(t_0)\neq 0.
\end{equation}

From equation (\ref{EqFor_fk2}) for finding $f_k$ it easily follows that the criterion for the uniqueness of the solution of the inverse problem for the subdiffusion equation has a similar form:
\begin{equation}\label{criterion}
b_{k,\rho}(t_0)\neq 0.
\end{equation}

	Let us establish two-sided estimates for $b_{k,\rho}(t_0)$ .
	First we suppose that $g(t)$ does not change sign (for the diffusion equation, i.e. for $b_{k,1}(t_0)$, see Sabitov et al. \cite{Sab}, \cite{Sab2}). Then $K_{0,\rho}$ is empty.
	
	\begin{lem}\label{invvv1} Let $g(t)\in C[0,T]$ and $g(t)\neq 0$, $t\in [0,T]$. Then there are constants $C_0,C_1>0$, depending on $t_0$, such that for all $k$:
		\[
		\frac{C_0}{\lambda_k}\leq |b_{k,\rho}(t_0)|\leq\frac{C_1}{\lambda_k}.
		\]
	\end{lem}
\begin{proof} By virtue of the Weierstrass theorem, we have $|g(t)|\geq g_0=const >0$. Apply the mean value theorem and Lemma \ref{MLint} to obtain 
	\[
	|b_{k,\rho}(t_0)| =\bigg|\int\limits _0^{t_0} \eta^{\rho-1} E_{\rho, \rho} (-\lambda_k  \eta^\rho)g(t_0-\eta)d\eta\bigg|=
	\]
	\[=|g(\xi_k)| t_0^\rho E_{\rho, \rho+1} (-\lambda_k t_0^\rho ) , \quad \xi_k\in[0,t_0].\]
 It is easy to see, that
 \[
  E_{\rho, \rho+1}(-t)=t^{-1}(1-E_\rho(-t)).
 \]
	Therefore, using Lemma \ref{MLmonoton} and the estimate $|g(t)|\geq g_0$ one has
 \[
	|b_{k,\rho}(t_0)|=|g(\xi_k)|\frac{1}{\lambda_k }(1-E_\rho(-\lambda_kt_0^\rho)\geq  \frac{C_0}{\lambda_k}.
	\]
	
	Finally Lemma \ref{mll4} implies
\[
	|b_{k,\rho}(t_0)|\leq C\frac{|g(\xi_k)|t_0^\rho}{1+\lambda_k t_0^\rho} \leq C\frac{ \max\limits_{0\leq\xi \leq t_0}|g(\xi)|}{\lambda_k}\leq \frac{C_1}{\lambda_k}.
	\]	
\end{proof}

\begin{thm}\label{thmNotChange}Let $\rho\in (0,1]$, $g(t)\in C[0,T]$ and $g(t)\neq 0$, $t\in [0,T]$. Moreover let function $\varphi(x)$ satisfy condition  (\ref{ml3}) and $\psi(x)$ satisfy condition (\ref{ml4}). Then there exists a unique solution of the inverse problem (\ref{prob1})-(\ref{ad}):
		\begin{equation}\label{f_NotChange}
			f(x)=\sum\limits_{k=1}^\infty \frac{1}{b_{k,\rho}(t_0)}\left[ \psi_k-\varphi_k E_\rho(-\lambda_k  t_0)\right]v_k(x),
		\end{equation}
		\begin{equation}\label{u_NotChange}
			u(x,t)=\sum\limits_{k=1}^\infty\varphi_k E_{\rho}(-\lambda_kt^\rho)v_k(x)+\sum\limits_{k=1}^\infty \frac{b_{k,\rho}(t)}{b_{k,\rho}(t_0)}\left[ \psi_k-\varphi_k E_\rho(-\lambda_k  t_0)\right]v_k(x).
		\end{equation}
\end{thm}

For the diffusion equation ($\rho=1$), this theorem is proved only in cases where $\Omega$ is an interval on $\mathbb{R}$ (see \cite{Sab}) or a rectangle on $\mathbb{R}^2$  (see \cite{Sab2}). It is a new theorem for subdiffusion equations ($\rho\in (0,1)$).
\begin{proof} 

Since $b_{k,\rho}(t_0)\neq 0 $ for all $ k \in \mathbb{N}$, then we get the following equations from (\ref{EqFor_fk2}):
\begin{equation}\label{inv4}
f_k=\frac{1}{b_{k,\rho}(t_0)}\left[ \psi_k-\varphi_k E_\rho(-\lambda_k  t_0)\right],
\end{equation}
\begin{equation}\label{inv5}
u_k(t)=\varphi_k E_{\rho}(-\lambda_kt^\rho)+\frac{b_{k,\rho}(t)}{b_{k,\rho}(t_0)}\left[ \psi_k-\varphi_k E_\rho(-\lambda_k  t_0)\right].
\end{equation}

With these Fourier coefficients, we have the following series for the unknown functions $f(x)$ and $u(x,t)$:
\begin{equation}\label{inv10}
f(x)=\sum\limits_{k=1}^\infty \frac{1}{b_{k,\rho}(t_0)}\left[ \psi_k-\varphi_k E_\rho(-\lambda_k  t_0)\right]v_k(x)=\sum\limits_{k=1}^\infty [f_{k,1}+f_{k,2}]v_k(x),
\end{equation}
\begin{equation}\label{inv11}
u(x,t)=\sum\limits_{k=1}^\infty\varphi_k E_{\rho}(-\lambda_kt^\rho)v_k(x)+\sum\limits_{k=1}^\infty \frac{b_{k,\rho}(t)}{b_{k,\rho}(t_0)}\left[ \psi_k-\varphi_k E_\rho(-\lambda_k  t_0)\right]v_k(x).
\end{equation}
If $F_j(x)$ is the partial sums of series (\ref{inv10}), then applying Lemma \ref{ml1}   as above, we have
\begin{equation}\label{inv6}
||\hat{A}^{-\sigma}F_j(x)||_{C(\Omega)}^2\leq \sum\limits_{k=1}^j \lambda_k^{2\sigma }|f_{k,1}+f_{k,2}|^2 \leq 2\sum\limits_{k=1}^j\lambda_k^{2\sigma }f_{k,1}^2+2\sum\limits_{k=1}^j\lambda_k^{2\sigma }f_{k,2}^2 \equiv 2I_{1,j}+2I_{2,j},
\end{equation}
where $\sigma > \frac{N}{4}$. Therefore by 
Lemma  \ref{invvv1} one has
\begin{equation}\label{inv7}
I_{1,j}\leq \sum\limits_{k=1}^j \frac{\lambda_k^{2\sigma}}{|b_{k,\rho}(t_0)|^2} |\psi_k|^2 \leq C \sum\limits_{k=1}^j \lambda_k^{\tau+2}|\psi_k|^2, \quad \tau=2\sigma>\frac{N}{2},
\end{equation}
\begin{equation}\label{inv27}
I_{2,j}\leq \sum\limits_{k=1}^j \left|\frac{E_\rho(-\lambda_k  t_0)}{b_{k,\rho}(t_0)}\right|^2\lambda_k^{2\sigma} |\varphi_k |^2 \leq C \sum\limits_{k=1}^j \lambda_k^{\tau}|\varphi_k|^2, \quad \tau=2\sigma>\frac{N}{2}.
\end{equation}

Thus, if $\varphi(x)$ satisfies conditions  (\ref{ml3}) and $\psi (x)$ satisfies conditions (\ref{ml4}), then from estimates of $I_{i,j}$ and (\ref{inv6}) we obtain $f(x) \in C(\overline{\Omega})$. Further, the fact that function $u(x,t)$ given by  (\ref{inv11}) is a solution to the inverse problem is proved exactly as the proof of Theorem \ref{mainForward}.

 To prove the uniqueness of the solution, assume the contrary, i.e., there are two different solutions $\{u_1,f_1\}$ and $\{u_2,f_2\}$ satisfying the inverse problem (\ref{prob1} )-(\ref{ad}). We need to show that $u\equiv u_1-u_2 \equiv 0$, $f\equiv f_1-f_2\equiv 0$. For $\{u,f\}$ we have the following problem:
 \begin{equation}\label{prob20}
\left\{
\begin{aligned}
& D_t^\rho u(x,t)-\Delta u(x,t) =f(x)g(t),\quad (x,t)\in \Omega \times(0,T],\\
&u(x,t)|_{\partial\Omega}=0,  \\
&u(x,0)=0, \quad x \in \Omega,\\
&u(x,t_0)=0, \quad x \in \Omega, \quad t_0 \in (0,T].
\end{aligned}
\right.
\end{equation} 
We take any solution $\{u,f\}$ and define $u_k=(u,v_k)$ and $f_k=(f,v_k)$. Then, due to the self-adjointness of the operator $ -\Delta$ and continuity of the derivatives of the solution up to the boundary of the domain $\Omega$, we have
\[
D_t^\rho u_k(t)= (D_t^\rho u, v_k)= (\Delta u, v_k)+f_k g(t)=( u,\Delta v_k)+f_k g(t)=-\lambda_k u_k(t)+f_k g(t).
\]
Therefore, for $u_k$ we obtain the Cauchy problem
\[
D_t^\rho u_k(t)+\lambda_k u_k(t) =f_kg(t),\quad t>0,\quad u_k(0)=0, 
\]
and the additional condition
\[
\quad u_k(t_0)=0.
\]
If $f_k$ is known, then the unique solution of the Cauchy problem has the form
\[
u_k(t)= f_k\int\limits_0^t \eta^{\rho-1} E_{\rho, \rho} (-\lambda_k \eta^\rho) g(t-\eta) d\eta=f_kb_{k,\rho}(t).
\]
Apply the additional condition to get
\[
u_k(t_0)= f_kb_{k,\rho}(t_0)=0.
\]
Since $b_{k,\rho}(t_0) \neq 0$ for all $k \in \mathbb{N} $, then due to completeness of the set of eigenfunctions $\{v_k\}$ in $L_2(\Omega)$, we finally have  $f(x)\equiv 0$ and  $u(x,t)\equiv0$. \end{proof} 

Now consider the case when $g(t)$ changes sign. In this case, function $b_{k,\rho}(t_0)$ can
become zero, and as a result, the set $K_{0, \rho}$ may turn out to be non-empty. Now we should consider separately the case of diffusion ($\rho=1$) and subdiffusion ($0<\rho<1$) equations.
	\begin{lem}\label{lemmaClassic}
	Let $\rho=1$, $g(t)\in C^1[0, T]$ and $g(t_0)\neq 0$. Then there are constants $C_0,C_1>0$, depending on $t_0$, such that for all $k\in K_1$ one has
		\[
	\frac{C_0}{\lambda_k}\leq |b_{k,1}(t_0)|\leq\frac{C_1}{\lambda_k}.
	\]
	\end{lem}
 \begin{proof}By integrating by parts and  the mean value theorem, we get
		\[
		b_{k,1}(t_0)=\int\limits_0^{t_0} e^{-\lambda_k  s} g(t_0-s)ds  =-\frac{1}{\lambda_k}g(t_0-s)  e^{-\lambda_k  s}\bigg|^{t_0}_0 - \frac{1}{\lambda_k}\int\limits_0^{t_0} e^{-\lambda_k  s} g'(t_0-s)ds = 
		\]
		\[
		=\frac{1}{\lambda_k} \big[g(t_0)-g(0)e^{-\lambda_k t_0}\big] +  \frac{g'(\xi_k)}{\lambda^2_k}\big[e^{-\lambda_k t_0}
-1\big], \quad \xi_k\in [0, t_0].		
\]
The first square bracket cannot vanish starting from some number $k$. Therefore, there exists a constant $C_0$ such that the required lower bound holds.

The upper estimate follows from the boundedness of function $g(t)$.
	\end{proof}

 In case of subdiffusion equation ($\rho\in (0,1)$) we have
	\begin{lem}\label{lemmaSub}Let $\rho\in (0,1)$, $g(t)\in C^1[0, T]$ and $g(0)\neq 0$. Then for sufficiently small $t_0$ there are constants $C_0,C_1>0$, such that for all $k\in K_\rho$:
		\[
		\frac{C_0}{\lambda_k}\leq |b_{k,\rho}(t_0)|\leq\frac{C_1}{\lambda_k}.
		\]
	\end{lem}
 \begin{proof}
	Let $\rho\in (0,1)$. Using equality (\ref{MLintFormula}) we integrate by parts, then apply the mean value theorem. Then we have
		\[	
		b_{k,\rho}(t_0)=\int\limits_0^{t_0}g(t_0-s) s^{\rho-1} E_{\rho, \rho} (-\lambda_k  s^\rho )  ds=\int\limits_0^{t_0}g(t_0-s) d\big[ s^{\rho} E_{\rho, \rho+1} (-\lambda_k  s^\rho ) \big] =
		\]
		\[
		=g(t_0-s)  s^{\rho} E_{\rho, \rho+1} (-\lambda_k  s^\rho )\bigg|^{t_0}_0+\int\limits_0^{t_0}g'(t_0-s)  s^{\rho} E_{\rho, \rho+1} (-\lambda_k  s^\rho )ds=
		\]
		\[
		=g(0)\,  t_0^{\rho} \,E_{\rho, \rho+1} (-\lambda_k  t_0^\rho)+ g'(\xi_k) \int\limits_0^{t_0} s^{\rho} E_{\rho, \rho+1} (-\lambda_k  s^\rho )ds, \quad \xi_k\in [0, t_0].
		\]
		For the last integral formula (\ref{MLintFormula}) implies
		\[
		\int\limits_0^{t_0} s^{\rho} E_{\rho, \rho+1} (-\lambda_k  s^\rho )ds=t_0^{\rho+1} E_{\rho, \rho+2}(-\lambda_k t_0^\rho).
		\]
		Apply the asymptotic estimate of the Mittag-Leffler functions (Lemma \ref{ml8}) to get
\[
b_{k,\rho}(t_0)=\frac{g(0)}{\lambda_k} +\frac{g'(\xi_k)}{\lambda_k} t_o + O\bigg(\frac{1}{(\lambda_k t_0^\rho)^2}\bigg).
\]
If  $g(0)\neq 0$, then for sufficiently small $t_0$ we obtain the required lower bound. This also implies the required upper bound.
\end{proof}	

Theorem \ref{thmNotChange} proves the existence and uniqueness of a solution to the inverse problem (\ref{prob1})-(\ref{ad}) under condition $g(t)\in C[0,T]$ and $g(t)\neq 0$, $t\in [0,T]$, i.e., $g(t)$ does not change sign. In Example 1, we saw that if this condition is violated, then the uniqueness of the solution to problem (\ref{prob1})-(\ref{ad}) is violated. Naturally, questions arise: if $g(t)$ changes sign, is uniqueness always violated? What can be said about the existence of a solution? How many solutions can there be?

	It should be emphasized that the answers to these questions were not known even for the classical diffusion equation (i.e. $\rho=1$).

	Lemmas \ref{lemmaClassic} and \ref{lemmaSub} proved above allow us to answer these questions. Let us formulate the corresponding result.

 \begin{thm}Let $g(t)\in C^1[0,T]$,  function $\varphi(x)$ satisfy condition  (\ref{ml3}) and $\psi(x)$ satisfy condition (\ref{ml4}). Further, we will assume that for $\rho=1$ the conditions of Lemma \ref{lemmaClassic} are satisfied, and for $\rho\in (0,1)$, the conditions of Lemma \ref{lemmaSub} are satisfied.
		
		1) If set $K_{0,\rho}$ is empty, i.e. $b_{k,\rho}(t_0)\neq 0$, for all $k$, then there exists a unique solution of the inverse problem (\ref{prob1})-(\ref{ad}):
		\begin{equation}\label{K0empty_f}
			f(x)=\sum\limits_{k=1}^\infty \frac{1}{b_{k,\rho}(t_0)}\left[ \psi_k-\varphi_k E_\rho(-\lambda_k  t_0)\right]v_k(x),
		\end{equation}
		\begin{equation}\label{K0empty_u}
			u(x,t)=\sum\limits_{k=1}^\infty\varphi_k E_{\rho}(-\lambda_kt^\rho)v_k(x)+\sum\limits_{k=1}^\infty \frac{b_{k,\rho}(t)}{b_{k,\rho}(t_0)}\left[ \psi_k-\varphi_k E_\rho(-\lambda_k  t_0)\right]v_k(x).
		\end{equation}
		
		2) If set $K_{0,\rho}$ is not empty, then for the existence of a solution to the inverse problem, it is necessary and  sufficient that the following  conditions
		\begin{equation}\label{ortogonal}
			\psi_k=\varphi_k  E_\rho(-\lambda_k t_0),\,\, k\in K_{0,\rho},
		\end{equation}
		be satisfied. In this case, the solution to the problem (\ref{prob1})-(\ref{ad}) exists, but is not unique:
		
		\begin{equation}\label{K0notempty_f}
			f(x)=\sum\limits_{k\in K_\rho} \frac{1}{b_{k,\rho}(t_0)}\left[ \psi_k-\varphi_k E_\rho(-\lambda_k  t_0)\right]v_k(x)+\sum\limits_{k \in K_{0,\rho}} f_k v_k(x),
		\end{equation}
		\begin{equation}\label{K0notempty_u}
			u(x,t)=\sum\limits_{k=1}^\infty\big[\varphi_k E_{\rho}(-\lambda_kt^\rho)+f_k\big]v_k(x),
		\end{equation}
		where $f_k$, $k\in K_{0,\rho}$, are arbitrary real numbers, such that
		\begin{equation}\label{adfk}
			\sum\limits_{k \in K_{0,\rho}} \lambda_k^{2\sigma}|f_k|^2<\infty,\,\, \sigma>\frac{N}{4}.
		\end{equation}
  \end{thm}
  \begin{proof}
      The proof of the first part of the theorem is completely analogous to the proof of Theorem \ref{thmNotChange}. As regards the proof of the second part of the theorem, we note the following.

If $k\in K_\rho$, then again from (\ref{EqFor_fk2}) we have (\ref{inv4}) and (\ref{inv5}). 

If $k\in K_{0,\rho}$, i.e. $b_{k,\rho}(t_0)=0$, then the solution of equation (\ref{EqFor_fk2}) with respect to $f_k$ exists if and only if the conditions (\ref{ortogonal}) are satisfied. In this case, the solution of the equation can be arbitrary numbers $f_k$, and in order for the functions (\ref{K0notempty_f}) and (\ref{K0notempty_u}) to have the necessary continuous derivatives, these numbers must, as it was in 1), satisfy condition (\ref{adfk}).
 \end{proof}
 
 Note that condition (\ref{ortogonal}) is rather difficult to verify. Given relation $E_\rho(-t)\neq 0$, $t>0$ (see Lemma \ref{MLmonoton}), one can replace this condition with a simpler condition.

\begin{remark}For conditions (\ref{ortogonal}) to be satisfied, it suffices that the following orthogonality conditions hold: 
\[
\varphi_k=(\varphi, v_k)=0,\,\, \psi_k=(\psi, v_k)=0,\,\, k\in K_{0,\rho}.
\]
\end{remark}

Let us briefly note some known results on inverse problems for the diffusion equation (i.e., $\rho =1$). In the work of D.G. Orlovskii \cite{Orl} abstract diffusion equations in Banach and Hilbert spaces are considered.  In the case of a Hilbert space, the elliptic part of the equation is self-adjoint, and the found uniqueness criterion is similar to (\ref{criterion}). A condition on the function $b_{k,1}(T)$ is found, which ensures the existence of a generalized solution (note that here condition (\ref{criterion}) is given at the point $t_0=T$).

In I.V. Tikhonov, Yu.S. Eidel'man \cite{Tix}, abstract diffusion equations in Banach and Hilbert spaces are also considered. In the case of a self-adjoint elliptic part, the uniqueness criterion coincides with (\ref{criterion}). It is shown that if we consider equations in a Banach space, then condition (\ref{criterion}) is not a criterion, and an addition to (\ref{criterion}) is found that turns (\ref{criterion}) into a uniqueness criterion for equations with a non-conjugate elliptic part.

The elliptic part of the diffusion equation in work A.I. Prilepko, A.B. Kostin \cite{Pr} is a second-order differential expression. Both non-self-adjoint and self-adjoint elliptic parts are considered. In this paper, $g(t)$ also depends on the spatial variable: $g(t):= g(x,t)$. In the case of a self-adjoint elliptic part, the authors managed to find a criterion for the uniqueness of the generalized solution of the inverse problem: the solution is unique if and only if the system
\[
w_k(x)=v_k(x)\int\limits_0^{t_0} g(x,t) e^{-\lambda_k(t_0-t)} dt,\quad k=1,2,\cdots
\]
is complete in $L_2(\Omega)$. It is easy to see that if $g(x,t)$ does not depend on $x$, then this criterion coincides with (\ref{criterion}). It should be emphasized that the Fourier method is not applicable to the equation considered in this paper.

The closest to our research are the works of K.B. Sabitov and A.R. Zaynullov \cite{Sab} and \cite{Sab2}. We borrowed some ideas from these works. In work \cite{Sab} the elliptic part of the equation is $u_{xx}$ defined on an interval (in \cite{Sab2} the Laplace operator on the rectangle). Having considered the over-determination condition in the form (\ref{ad}), it is shown that the criterion for the uniqueness of the classical solution is (\ref{criterion}). When condition (\ref{criterion}) is satisfied, a classical solution is constructed by the Fourier method. We note that the existence of a classical solution was not discussed in the works listed above.

\section{Conclusion}

In this paper, we consider the subdiffusion equation with a fractional derivative of order $\rho\in (0,1]$, and take the Laplass operator as the elliptic part. The right-hand side of the equation has the form $f(x) g(t)$, where $g(t)$ is a given function and the inverse problem of determining function $f(x)$ is considered. Following the works \cite{Pr} and \cite{Sab}, the over-determination condition is taken in a more general form. It is proved that the criterion for the uniqueness of the classical solution of the inverse problem for the subdiffusion equations coincides with the analogous condition for the diffusion equations.

In the case when this condition is not satisfied, a necessary and sufficient condition for the existence of a classical solution is found and all solutions of the inverse problem are constructed using the classical Fourier method. We emphasize that this part of the work is also new for the classical diffusion equation.

The results of this work can be generalized to more general subdiffusion equations by replacing the Laplace operator in (\ref{prob1}) with a high-order self-adjoint elliptic operator with variable coefficients. At the same time, instead of the result of V.A. Il'in should be used with similar results of Sh.A. Alimov (see e.g. \cite{Al12}) for a general elliptic operator.

\section{Acknowledgement}
The authors are grateful to Sh. A. Alimov for discussions of
these results.
The authors acknowledge financial support from the  Ministry of Innovative Development of the Republic of Uzbekistan, Grant No F-FA-2021-424.

\end{document}